\documentclass{amsart}
\usepackage{amsmath,amssymb,amsthm}
\usepackage{cite}

\numberwithin{equation}{section}

\newtheorem{Theorem}{Theorem}[section]

\newtheorem{Corollary}[Theorem]{Corollary}
\newtheorem{Lemma}[Theorem]{Lemma}
\newtheorem{Proposition}[Theorem]{Proposition}

\makeatletter
\def\section{\@startsection {section}{1}{\z@}%
{2.5ex plus - 1ex minus -.2ex}%
{.5ex \@plus.3ex}%
{\large\bfseries}%
}
\def\subsection{\@startsection {subsection}{1}{\z@}%
{2.5ex plus - 1ex minus -.2ex}%
{.5ex \@plus.3ex}%
{\normalsize\bfseries}%
}
\makeatother

\begin{document}

\title[Blow up for a system of evolution equations]{
Blow-up of solutions for weakly coupled systems
of complex Ginzburg-Landau equations}

\author[K. Fujiwara]{Kazumasa Fujiwara}

\address[K. Fujiwara]{%
Department of Pure and Applied Physics,
Waseda University,
3-4-1, Okubo, Shinjuku-ku,
Tokyo, 169-8555, Japan}

\email{k-fujiwara@asagi.waseda.jp}

\thanks{
The first author was partly supported by the Japan Society for the Promotion of Science,
Grant-in-Aid for JSPS Fellows no 16J30008.
}

\author[M. Ikeda]{Masahiro Ikeda}
\address[M. Ikeda]{Department of Mathematics, Faculty of Science and Technology,
Keio University,
3-14-1 Hiyoshi, Kohoku-ku, Yokohama, 223-8522, Japan/Center for Advanced Intelligence Project, RIKEN, Japan}
\email{masahiro.ikeda@keio.jp/masahiro.ikeda@riken.jp}

\author[Y. Wakasugi]{Yuta Wakasugi}

\address[Y. Wakasugi]{Department of Engineering for Production and Environment,
Graduate School of Science and Engineering,
Ehime University,
3 Bunkyo-cho, Matsuyama, Ehime, 790-8577,
Japan}%
\email{wakasugi.yuta.vi@ehime-u.ac.jp}

\begin{abstract}
Blow-up phenomena of
weakly coupled systems of several evolution equations,
especially complex Ginzburg-Landau equations
is shown by a straightforward ODE approach
not so-called test-function method used in \cite{bibx:35},
which gives the natural blow-up rate.
The difficulty of the proof is that,
unlike the single case,
terms which come from the fact that the Laplacian cannot be absorbed
into the weakly coupled nonlinearities.
A similar ODE approach is applied to heat systems by Mochizuki \cite{bibx:30}
to obtain the lower estimate of lifespan.
\end{abstract}

\maketitle

\section{Introduction}
In the present paper,
we study a blow-up phenomena
(blow-up rate and estimates of lifespan)
for the following Cauchy problem for the Ginzburg-Landau systems
with the weakly coupled nonlinearity
by developing an ODE approach used in \cite{bibx:30}:
	\begin{align}
	\begin{cases}
	\partial_t u + \alpha_1 \Delta u = \beta_1 |v|^p,
	& t \in \lbrack 0, T), \quad x \in \mathbb X,\\
	\partial_t v + \alpha_2 \Delta v = \beta_2 |u|^q,
	& t \in \lbrack 0, T), \quad x \in \mathbb X,\\
	u(0,x) = u_0(x),\quad
	v(0,x) = v_0(x),
	& x \in \mathbb X,
	\end{cases}
	\label{eq:1.1}
	\end{align}
where
$u=u(t,x)$ and $v=v(t,x)$ are unknown complex-valued functions of $(t,x)$,
$\alpha_1,\alpha_2,\beta_1,\beta_2 \in \mathbb C \backslash \{0\}$ and
$p, q \geq 1$ are constants,
$\mathbb X$ denotes the $n$-dimensional Euclidean space $\mathbb{R}^n$
or Torus $\mathbb{T}^n$
with $n \in \mathbb N$,
and $\Delta:=\sum_{i=1}^n\frac{\partial^2}{\partial x_i^2}$
denotes the Laplacian on $\mathbb X$.
$T$ denotes the maximal existence time of the function $(u,v)$
and is called lifespan.

Our aim in the present paper is to understand
a blow-up mechanism of the weakly-coupled system \eqref{eq:1.1}
by an ODE argument which is more direct than the test-function method
used in \cite{bibx:35}.

When $\alpha_1=\alpha_2=-1$, $\beta_1=\beta_2=1$,
$u$ and $v$ are real-valued and $u_0$ and $v_0$ are non-trivial non-negative,
the problem \eqref{eq:1.1} becomes the Cauchy problem
for the following heat systems with the weakly coupled nonlinearities:
	\begin{align}
	\begin{cases}
	\partial_t u - \Delta u = v^p,
	& t \in \lbrack 0, T), \quad x \in \mathbb X,\\
	\partial_t v - \Delta v = u^q,
	& t \in \lbrack 0, T), \quad x \in \mathbb X,\\
	u(0,x) = u_0(x) \geq 0,\quad
	v(0,x) = v_0(x) \geq 0,
	& x \in \mathbb X.
	\end{cases}
	\label{eq:1.2}
	\end{align}
We extend the blow-up result for \eqref{eq:1.2} into the complex setting.

We, at first, recall several previous results about blow-up of single heat equations.
The following Cauchy problem for the single heat equation with a power-type nonlinearity
has been extensively studied:
	\begin{align}
	\begin{cases}
	\partial_t u - \Delta u = u^p,
	&t \in \lbrack 0,T),\quad x \in \mathbb R^n,\\
	u(0,x) = u_0(x),
	&x \in \mathbb R^n,
	\end{cases}
	\label{eq:1.3}
	\end{align}
where $u$ is unknown non-negative function,
$u_0$ is non-negative functions which is not identically $0$,
and $p \geq 1$.
As a pioneering work,
Fujita \cite{bibx:9} showed that
if $p < p_F := 1 + 2/n$,
where $p_F$ is called the Fujita exponent,
then any non-trivial non-negative solutions blow up in a finite time
by using the contradiction argument.
See also \cite{bibx:10}.
He used the following ODE
coming from \eqref{eq:1.3} without the Laplacian to show the blow-up result:
	\begin{align}
	\frac{d}{dt} f(t) = f(t)^p,
	\label{eq:1.4}
	\end{align}
where $f$ is a positive $C^1$-function.
If $f(0) > 0$, then the solution of \eqref{eq:1.4} is given by
	\[
	f(t) = (f(0)^{-p+1}- (p-1) t)^{-\frac{1}{p-1}}
	\]
and therefore $f$ blows up at $t = \frac{f(0)^{-p+1}}{p-1}$
and the blow-up rate is $-\frac{1}{p-1}$.
Later, Hayakawa \cite{bibx:14}
and Kobayashi, Sirao, and Tanaka \cite{bibx:22} independently showed that
in the critical case where $p = p_F$, any non-trivial non-negative solutions
blow up in a finite time.
Moreover, Giga and Kohn \cite{bibx:13} proved that
if $p > 1$ when $n=1,2$ and $1 < p < \frac{n + 2}{n-2}$ when $n \geq 3$, then
the positive solution $u$ of \eqref{eq:1.3} blows up in type I rate, i.e.
	\begin{align}
	\| u(t) \|_{L^{\infty}(\mathbb R^n)} \leq C(T-t)^{-\frac{1}{p-1}},
	\quad t \in \lbrack 0, T)
	\label{eq:1.5}
	\end{align}
for some constant $C$ independent of $t$ and $T$.
The estimate \eqref{eq:1.5} implies that
behavior of blow-up solutions of \eqref{eq:1.3}
and that of \eqref{eq:1.4} have similar blow-rate and lifespan.
Later, Matano and Merle \cite{bibx:26} extended the result of \cite{bibx:13}
to $\frac{n+2}{n-2} < p < p_{JL}$,
where $p_{JL}$ is the Joseph-Lundgren exponent given by
$p_{JL} := \infty \ (3\le n \le 10)$,
$p_{JL} := 1+\frac{4}{n-4-\sqrt{n-1}} \ (n \ge 11)$.
We remark that Herrero and Vel\'azquez \cite{bibx:16} showed that
for $n \geq 11$ and $p_{JL} < p$,
there exists a blow-up solution of \eqref{eq:1.3}
which violates the estimate \eqref{eq:1.5}.
Thus the analogy from ODE \eqref{eq:1.4}
does not work to the heat equation \eqref{eq:1.3}
in the case where the power and spacial dimension are sufficiently high.

In the subcritical case where $p < p_F$,
Lee and Ni \cite{bibx:25} obtained
the sharp estimate of lifespan to \eqref{eq:1.3}
	\[
	c \epsilon^{-\frac{1}{\frac{1}{p-1}-\frac{n}{2}}}
	\leq T \leq
	C \epsilon^{-\frac{1}{\frac{1}{p-1}-\frac{n}{2}}},
	\]
where $\epsilon > 0$ is sufficiently small
and $u_0$ is replaced by $u_0 = \epsilon u_0$.
We note that
they used the comparison principle for heat equations
and introduced blow-up sub-solutions and super-solutions.
We remark that the exponent
$(\frac{1}{p-1}-\frac{2}{n})^{-1}$
is sharp with respect to the size of the initial data.
For more information about \eqref{eq:1.3},
see \cite{bibx:8,bibx:15,bibx:27,bibx:29}
and the references therein.

On the other hand,
Zhang \cite{bibx:35} studied
non-existence of global weak solutions
of a semilinear parabolic equations for some initial data
and a nonlinearity of a variable coefficient
by using the so-called test function method.
His method is based on a contradiction argument
with a weak form.
If we apply his method to \eqref{eq:1.3},
the weak form is
	\begin{align*}
	&\int_{0}^{T} \int_{\mathbb R^n}
	u(t,x) (- \partial_t \phi(t,x) - \Delta \phi(t,x))
	\thinspace dx \thinspace dt\\
	&= \int_{\mathbb R^n} u_0(x) \phi(x)
	+\int_{0}^{T} \int_{\mathbb R^n}
	|u(t,x)|^p \phi(t,x)
	\thinspace dx \thinspace dt,
	\end{align*}
where $\phi$ is a smooth, non-negative, compact supported function
(see also \cite{bibx:28}).
By modifying his argument,
Kuiper \cite{bibx:24} gave
an estimate of lifespan for some parabolic equations
for some slowly decreasing initial datum.
His argument has been applied to semilinear Schr\"odinger
and damped wave equations.
However,
the test function method does not give blow-up rate
since, roughly speaking, unknown functions are canceled out in this argument.

Next we recall several previous results of blow-up
for single complex Ginzburg-Landau equations,
	\begin{align}
	\begin{cases}
	\partial_t u - \alpha \Delta u = F(u),
	& t \in \lbrack 0, T), \quad x \in \mathbb X,\\
	u(0,x) = u_0(x),
	& x \in \mathbb X,
	\end{cases}
	\label{eq:1.6}
	\end{align}
where $u$ is complex-valued unknown function,
$\alpha \in \mathbb C \backslash \{0\}$,
and $F$ is the nonlinearity.
Ogawa and Tsutsumi \cite{bibx:31} studied
blow-up for the case where $i \alpha < 0$, $F(u)= i |u|^4u$, $\mathbb X = \mathbb T$.
Later, Ozawa and Yamazaki \cite{bibx:33} studied \eqref{eq:1.6}
in the case where $\mathrm{Re} \thinspace \alpha \geq 0$,
$F(u) = (\kappa + i \beta) |u|^p + \gamma u$ with $\beta >0$,
$\kappa, \gamma \in \mathbb R$,
and
	\[
	\mathrm{Im} \int_{\mathbb T} u_0(x) dx > 0.
	\]
They introduced
	\[
	M(t) =
	\frac{e^{-\gamma t}}{2\pi} \mathrm{Im}
	\int_{\mathbb T} u(t,x) dx,
	\quad 0 \leq t < T
	\]
and showed that $M$ is positive
and satisfies the ordinary differential inequality (ODI)
	\begin{align}
	\frac{d}{dt} M(t) \geq \frac{e^{(p-1)\gamma t}}{2 \pi} M(t)^p,
	\quad 0 \leq t < T,
	\label{eq:1.7}
	\end{align}
which implies that $M$ blows up at a finite positive time
by a comparison principle.
We remark that in their argument,
the embedding $L^p(\mathbb T^n) \hookrightarrow L^1(\mathbb T^n)$
and identity
	\[
	\int_{\mathbb T^n} \Delta u(t,x) dx = 0
	\]
for $n \geq 1$ play a crucial role to obtain the ODI \eqref{eq:1.7}.
Oh \cite{bibx:32} studied blow-up of \eqref{eq:1.6}
in the case where $\alpha \in \mathbb C \backslash \{0\}$,
$F(u) = \lambda |u|^p$,
and $\mathbb X = \mathbb T^n$ with $n \geq 1$
by the test function method.
He showed that if
	\begin{align}
	\mathrm{Re} \thinspace \lambda
	\ \mathrm{Im} \int_{\mathbb T^n} u_0(x) dx < 0,
	\quad \mbox{or} \quad
	\mathrm{Im} \lambda \ \mathrm{Re} \int_{\mathbb T^n} u_0(x) dx > 0,
	\label{eq:1.8}
	\end{align}
then there is no global weak solutions.
The second and third authors \cite{bibx:20} applied the test function method
to study blow-up of \eqref{eq:1.6} in the case where $\alpha = i$, $F(u) = \lambda |u|^p$,
and $\mathbb X = \mathbb R^n$ with $n \geq 1$.
They showed if $1 < p \leq 1 + \frac{2}{n}$
and $u_0 \in L^1(\mathbb R^n) \cap L^2(\mathbb R^n)$ satisfies
	\[
	\mathrm{Re} \lambda \ \mathrm{Im} \int_{\mathbb R^n} u_0(x) dx < 0,
	\quad \mbox{or} \quad
	\mathrm{Im} \lambda \ \mathrm{Re} \int_{\mathbb R^n} u_0(x) dx > 0,
	\]
then the solution blows up in a finite time.
Moreover, by modifying the test function method of Kuiper,
the second author and Inui \cite{bibx:18,bibx:19} obtained
an upper bound of lifespan of solutions for the same problem
with $1 < p < 1 + \frac{4}{n}$
for some initial datum which decay slowly or have singularity at the origin,
where $1+\frac{4}{n}$ is the scaling critical exponent of $L^2(\mathbb R^n)$.
Recently, the first author and Ozawa \cite{bibx:11,bibx:12}
extended the results of Oh and the second author and Inui \cite{bibx:18, bibx:19}.
They showed blow-up phenomena
(blow-up rate and upper bounds of lifespan of solutions)
by an ODE argument connected with the test function method.
Indeed,
in the torus case where $\mathbb X = \mathbb T^n$ and $p >1$,
for $L^1(\mathbb T^n)$-initial data,
let
	\[
	\widetilde M(t) = \mathrm{Re} \bigg( \overline \lambda \int_{\mathbb T^n} u(t,x) dx
	\bigg),
	\quad t \in \lbrack 0, T).
	\]
Then, as with the approach of Ozawa and Yamazaki \cite{bibx:33}, they showed that,
if
	\begin{align}
	\widetilde M(0)
	= \mathrm{Re} \bigg( \overline \lambda \int_{\mathbb T^n} u_0(x) dx
	\bigg)
	> 0,
	\label{eq:1.9}
	\end{align}
then $\widetilde M$ satisfies
	\begin{align*}
	\frac{d}{dt} \widetilde M(t)
	&= |\lambda|^{2} \int_{\mathbb T^n} |u(t,x)|^p dx\\
	&\geq |\mathbb T^n|^{-(p-1)} |\lambda|^{2-p}
	\widetilde M(t)^p,
	\quad t \in \lbrack 0, T)
	\end{align*}
and this ODI implies that $\widetilde M$ blows up at a finite time.
We remark that the condition \eqref{eq:1.9} includes the condition \eqref{eq:1.8}.
Moreover,
in the Euclidean case where $\mathbb X = \mathbb R^n$ and $1 < p < p_F$,
for $L^1(\mathbb T^n)$-initial data,
if $u_0 \in L^1(\mathbb R^n) \cap L^2(\mathbb R^n)$,
then solutions for \eqref{eq:1.6} satisfy the ODI,
	\begin{align}
	\frac{d}{dt} \bigg(
	\int_{\mathbb R^n} u(t,x) \phi(x) dx
	- \frac{1}{2} U_0 \bigg)
	\geq C
	U_0^{\frac{n(p-1)}{\frac{2}{p-1}-n}}
	\bigg(
	\int_{\mathbb R^n} u(t,x) \phi(x) dx
	- \frac{1}{2} U_0 \bigg)^p
	\label{eq:1.10}
	\end{align}
for $0 < t < T$, where $C$ is some positive constant and
	\[
	U_0 = \int_{\mathbb R^n} u_0(x) \phi(x) dx
	\]
with a smooth test function $\phi$.
Therefore,
$\int_{\mathbb R^n} u(t,x) \phi(x) dx - \frac{1}{2} U_0$
is a super-solution of ODE \eqref{eq:1.4}
with a constant.
Thus the comparison principle implies that
the inequality,
	\begin{align*}
	\int_{\mathbb R^n} u(t,x) \phi(x) dx
	&\geq C_1 \bigg( U_0^{-p+1}
	- C_2 U_0^{\frac{n(p-1)}{\frac{2}{p-1}-n}} t
	\bigg)^{-\frac{1}{p-1}}\\
	&= C_1
	U_0^{-\frac{n}{\frac{2}{p-1}-n}}
	\bigg( U_0^{-\frac{1}{\frac{1}{p-1}-\frac{n}{2}}}
	- C_2 t
	\bigg)^{-\frac{1}{p-1}}
	\end{align*}
holds with some positive constants $C_1$ and $C_2$
for any $0 < t <T$.
In their argument,
the condition $p < p_F$ plays a crucial role
to obtain ODI \eqref{eq:1.10} by a scaling transformation of $\phi$.

Let us recall several previous results
for blow-up of weakly coupled heat systems \eqref{eq:1.2}.
Escobedo and Herrero \cite{bibx:3} showed that
if $p, q > 0$ and $pq > 1$,
then all positive solutions blow up
if
	\begin{align}
	\max \bigg( \frac{p+1}{pq-1},\ \frac{q+1}{pq-1} \bigg) \geq \frac{n}{2}
	\label{eq:1.11}
	\end{align}
and global solutions with small data
exist if
	\[
	\max \bigg( \frac{p+1}{pq-1},\ \frac{q+1}{pq-1} \bigg) < \frac{n}{2}.
	\]
We remark that if $p=q$, then \eqref{eq:1.11} coincides with $p \leq p_F$.
In this sense, their results correspond to the case
of single heat equation \eqref{eq:1.3}.
We remark that it is shown in \cite{bibx:2} that
even in the case where $\mathrm{min}(p,q) < 1$,
\eqref{eq:1.2} has local solutions.
Moreover,
Andreucci, Herrero, and Vel\'azquez \cite{bib:1} showed that
when $\mathbb X  = \mathbb{R}^n$,
the solution $(u,v)$ to \eqref{eq:1.2} satisfies
	\begin{align}
	u(t,x) \leq C (T-t)^{-\frac{p+1}{pq-1}},\quad
	v(t,x) \leq C (T-t)^{-\frac{q+1}{pq-1}}
	\label{eq:1.12}
	\end{align}
for $ 0 < t < T$
with a positive constant $C$
(see \cite{bibx:4} for bounded domain cases).
After that, for general unbounded domain
$\mathbb X = \Omega$,
Fila and Souplet \cite{bibx:7}
obtained the estimates \eqref{eq:1.12},
provided that
solution $(u,v)$ satisfies the stronger condition below than \eqref{eq:1.2},
	\begin{align}
	\max \bigg( \frac{p+1}{pq-1},\ \frac{q+1}{pq-1} \bigg) \geq \frac{n+1}{2}.
	\label{eq:1.13}
	\end{align}
We expect that when $\mathbb X = \Omega$ is an exterior domain,
the condition \eqref{eq:1.13} above can be relaxed to the condition \eqref{eq:1.11}.
For corresponding results of general component cases,
we refer the readers to \cite{bibx:6} and the references therein.
Next,
Mochizuki \cite{bibx:30} obtained the sharp estimate of lifespan
of the same problem \eqref{eq:1.2} with $p,q >1$
for non-negative initial data.
In order to obtain the upper bound of lifespan,
he used the following ODE systems:
	\begin{align}
	\begin{cases}
	\frac{d}{dt} f(t) + C_1 f(t) = C_2 g(t)^p,
	& t \in \lbrack 0, T),\\
	\frac{d}{dt} g(t) + C_1 g(t) = C_2 f(t)^q,
	& t \in \lbrack 0, T),\\
	f(0)=f_0, \quad g(0)=g_0
	\end{cases}
	\label{eq:1.14}
	\end{align}
with positive constants $C_1,C_2$ and positive numbers $f_0, g_0$.
He showed that the product $f(t)g(t)$ enjoys the following ODI:
	\begin{align}
	\frac{d}{dt} (f(t)g(t)) + C_1 f(t)g(t)
	\geq C_3 (f(t)g(t))^{\frac{(p+1)(q+1)}{p+q+2}},
	t \in \lbrack 0, T)
	\label{eq:1.15}
	\end{align}
with some positive constant $C_3$,
which implies that solutions of \eqref{eq:1.14} blow up at a finite time.
Since with small $\epsilon > 0$ and a solution $(u,v)$ for \eqref{eq:1.2},
	\[
	\bigg( \int_{\mathbb R^n} u(t,x) e^{-\epsilon |x|^2} dx,\quad
	\int_{\mathbb R^n} v(t,x) e^{-\epsilon |x|^2} dx \bigg)
	\]
is a super-solution of \eqref{eq:1.14},
the solution of \eqref{eq:1.2} blows up at a finite time
and the lifespan is estimated from above by that of solutions of \eqref{eq:1.14}.
Moreover,
he composed super-solutions of \eqref{eq:1.2}
which blow up in a finite time with the blow-up solutions of
	\begin{align}
	\begin{cases}
	\frac{d}{dt} f(t) = g(t)^p,
	& t \in \lbrack 0, T),\\
	\frac{d}{dt} g(t) = f(t)^q,
	& t \in \lbrack 0, T),\\
	f(0)=f_0,\quad g(0)=g_0.
	\end{cases}
	\label{eq:1.16}
	\end{align}
We note that $f$ and $g$ enjoy the following ODE:
	\begin{align}
	\frac{1}{q+1} \frac{d}{dt} f(t)^{q+1}
	= f(t)^p g(t)^q
	= \frac{1}{p+1} \frac{d}{dt} g(t)^{p+1},
	\quad t \in \lbrack 0, T).
	\label{eq:1.17}
	\end{align}
This identity divides ODE systems \eqref{eq:1.14} into two single ODEs.
The lifespan of solutions to the divided ODEs can be estimated.
In conclusion,
if $p > q > 1$,
for for the following initial data:
	\[
	(u_0(x),v_0(x))
	= (\epsilon \langle x \rangle^{-a},
	\epsilon^{\frac{2(q+1)-\min(b,n)(pq-1)}{2(p+1)
	- \min(a,n)(pq-1)}} \langle x \rangle^{-b}),
	\]
with small $\epsilon >0$ and some $a,b \neq n$ satisfying
$0< \min(a,n) < \frac{2(p+1)}{pq-1}$, $0< \min(b,n) < \frac{2(q+1)}{pq-1}$,
the lifespan is estimated by
	\[
	c \epsilon^{- \frac{1}{\frac{p+1}{pq-1}-\frac{n}{2}}}
	\leq T
	\leq C \epsilon^{- \frac{1}{\frac{p+1}{pq-1}-\frac{n}{2}}}
	\]
with some positive constants $c$ and $C$ independent of $\epsilon$.
See also \cite{bibx:5,bibx:23,bibx:34} and reference therein.

In the present paper,
we introduce an unified ODE approach to show the blow-up results
(blow-up rate and upper bound of lifespan)
for weakly coupled systems.
As an application,
we show the blow-up rate and upper bound of lifespan
for the complex Ginzburg-Landau systems
on both of torus and Euclidean spaces.

\begin{Theorem}
\label{Theorem:1.1}
Let $\mathbb X = \mathbb T^n$.
Let $\alpha_1, \alpha_2, \beta_1, \beta_2 \in \mathbb C \backslash \{0\}$.
Let $p,q > 0$ satisfying $pq > 1$ and $p \geq q$.
Let $(u_0, v_0)$ be not $(0,0)$ but $L^1(\mathbb T^n) \times L^1(\mathbb T^n)$ function.
Assume that there exists a solution
$(u,v) \in [ C^2((0,T)\times \mathbb T^n) ]^2$
to \eqref{eq:1.1} with the initial data $(u_0, v_0)$.
Then $(u,v)$ blows up in a finite time.
In particular, the estimates
	\begin{align}
	T \leq T_0 :=
	\bigg( \int_{\mathbb T^n} u_0(x) dx
	\bigg)^{-\frac{pq-1}{{p+1}}}
	\label{eq:1.18}
	\end{align}
and
	\begin{align}
	\int_{\mathbb T^n} u(t,x) dx
	&\geq C (T_0-t)^{- \frac{p+1}{pq-1}},
	\label{eq:1.19}\\
	\int_{\mathbb T^n} v(t,x) dx
	&\geq  C (T_0-t)^{- \frac{q+1}{pq-1}},
	\label{eq:1.20}
	\end{align}
hold for $t\in \lbrack 0,T)$ with some positive constant $C$ independent of $t$ and $T$.
\end{Theorem}

\begin{Theorem}
\label{Theorem:1.2}
Let $\mathbb X = \mathbb R^n$.
Let $\alpha_1, \alpha_2 < 0$ and $\beta_1, \beta_2 \in \mathbb C \backslash \{0\}$.
Let $p,q > 0$ satisfying $pq > 1$ and $p \geq q$.
Let $(u_0, v_0)$ be not $(0,0)$ and be non-negative
$L_{\mathrm{loc}}^1(\mathbb R^n) \times L_{\mathrm{loc}}^1(\mathbb R^n)$ function.
Assume that there exists a solution
$(u,v) \in [C^2((0,T)\times \mathbb R^n)]^2$
to \eqref{eq:1.1} with the initial data $(u_0, v_0)$.
Then $(u,v)$ blows up in a finite time.
In particular,
there is a compactly supported non-negative function
$\phi \in C^2(\mathbb R^n)$ such that the estimates
	\begin{align}
	T \leq T_1 := C
	\bigg( \int_{\mathbb R^n} u_0(x) \phi(x) dx
	\bigg)^{-\frac{1}{\frac{p+1}{pq-1} -\frac{n}{2}}}
	\label{eq:1.21}
	\end{align}
and
	\begin{align}
	\int_{\mathbb R^n} u(t,x) \phi(x) dx
	&\geq C (T_1-t)^{- \frac{p+1}{pq-1}},
	\label{eq:1.22}\\
	\int_{\mathbb R^n} v(t,x) \phi(x)dx
	&\geq  C (T_1-t)^{- \frac{q+1}{pq-1}},
        \label{eq:1.23}
	\end{align}
hold for $t\in \lbrack 0,T)$ with some positive constant $C$ independent of $t$ and $T$.
\end{Theorem}

We collect estimates of ODE systems \eqref{eq:1.14} and \eqref{eq:1.16}
in Section \ref{section:2}.
Then we give a proof of Propositions \ref{Theorem:1.1}
and \ref{Theorem:1.2} in Sections \ref{section:4} and \ref{section:3}, respectively.


\section{Preparation: ODE Arguments}
\label{section:2}

In this section,
we collect sub-solutions for some weakly coupled ODE systems
connected to \eqref{eq:1.1}.
In particular,
we show that solutions of weakly coupled systems of two ODEs
are larger than solutions of single ODEs with a modified nonlinearity.

Here, the following solutions for single ODEs
give the basis to obtain the explicit sub-solutions
of ODE systems.

\begin{Lemma}
Let $\rho > 1$, $\mu >0$, and $f_0 > 0$.
Set
	\begin{align*}
	T_f&:=\mu^{-1} (\rho-1)^{-1} f_0^{1-\rho}\\
	f(t) &:= \left\{ f_0^{1-\rho} - (\rho -1) \mu t \right\}^{-\frac{1}{\rho-1}},\ 
	\text{for}\ t\in \lbrack 0,T_f).
	\end{align*}
Then $f\in C^{\infty}((0,T_f))$ is the unique solution to the Cauchy problem for the ODE
	\[
	\begin{cases}
	f'(t) = \mu f(t)^\rho,\\
	f(0) = f_0.
	\end{cases}
	\]
\end{Lemma}

It is well known that
we can compare super-solutions and sub-solutions
for single ODEs with appropriate initial datum.
Since \eqref{eq:1.15} holds,
the comparison principle for singe ODE implies blow-up of \eqref{eq:1.14}
Moreover,
Kamke \cite{bibx:21} showed that
the comparison principle holds
for weakly coupled ODE systems.
In particular, the following Lemma \ref{Lemma:2.2} holds
and we can obtain lower bounds of each solutions
instead of sum of product of solutions.
We also refer the reader the text book of Hsu \cite{bibx:17}
and remark that their statements are more general than Lemma \ref{Lemma:2.2}.

\begin{Lemma}[{\cite[Theorem 2.6.3]{bibx:17}\cite[Satz 6]{bibx:21}}]
\label{Lemma:2.2}
Let $C \geq 0$, $T>0$,
and let $p,q>0$ satisfy $pq > 1$.
Let $\mathcal F, \mathcal G$ be non-decreasing functions on $\mathbb R$.
Let $f_1,f_2,g_1,g_2 \in C^1([0,T))$ satisfy ODIs,
	\[
	\begin{cases}
	\frac{d}{dt} f_2(t) + C f_2(t) \geq \mathcal G(g_2(t)),\\
	\frac{d}{dt} g_2(t) + C g_2(t) \geq \mathcal F(f_2(t)),
	\end{cases}
	\qquad
	\begin{cases}
	\frac{d}{dt} f_1(t) + C f_1(t) \leq \mathcal G(g_1(t)),\\
	\frac{d}{dt} g_1(t) + C g_1(t) \leq \mathcal F(f_1(t))
	\end{cases}
	\]
for $0 \leq t < T$
with the initial condition $f_2(0) > f_1(0)$ and $g_2(0) > g_1(0)$.
Then $f_2(t) > f_1(t)$ and $g_2(t) > g_1(t)$ for $0 \leq t < T$.
\end{Lemma}

For the reader's convenience,
we give a proof of the lemma.

\begin{proof}
On the contrary, we assume that there exists $t_0\in (0,T)$ such that $f_2(t_0) = f_1(t_0)$.
Then, we can define
	\[
	T_f = \min \{ 0 < t \le t_0 ;\ f_1(t) = f_2(t) \}.
	\]
Thus, the function
$\Phi(t) := e^{Ct}(f_2(t) - f_1(t))\in C^1((0,T))$
satisfies
$\Phi(t) > 0$ for $t \in \lbrack 0,T_f)$ and $\Phi (T_f) = 0$.
Hence, there exists $T_f' \in (0,T_f)$ such that
$\Phi'(T_f') < 0$.
On the other hand, by the ODIs, we see that
	\[
	\Phi'(t) = 
	\frac{d}{dt} \left\{ e^{Ct} (f_2(t) - f_1(t) )\right\}
	\geq e^{Ct} \{ \mathcal G(g_2(t)) - \mathcal G(g_1(t)) \},\ \text{for}\ t\in [0,T),
	\]
which leads to
$e^{CT_f'} \{ \mathcal G(g_2(T_f')) - \mathcal G(g_1(T_f')) \} < 0$.
This estimate and the monotonicity of
$\mathcal{G}$
imply
$g_2(T_f') < g_1(T_f')$.
Noting that
$g_2(0) > g_1(0)$
and $g_2$ and $g_1$ are continuous on $[0,T)$, there exists $t_1\in (0,T_f')$ such that $g_1(t_1)=g_2(t_1)$. Therefore we can also define
	\[
	T_g = \min \{ 0 < t \le t_1 ;\ g_1(t) = g_2(t) \}.
	\]
By the definitions of $T_f$, $T_f'$ and $T_g$, we obtain $T_f > T_f' > T_g$.
However, $T_g > T_f$ holds by the same argument, which leads to a contradiction. Therefore for any $t\in (0,T)$, $f_2(t) \ne f_1(t)$. By $f_2(0)>f_1(0)$ and the continuity of the functions $f_1$ and $f_2$, we have $f_2(t)>f_1(t)$ for $t\in [0,T)$. In the similar manner, we can prove $g_2(t)>g_1(t)$ for $t\in [0,T)$, which completes the proof of the lemma.
\end{proof}

In the following proposition,
we show explicit solutions for weakly coupled systems of ordinary differential equations,
which will be used to prove blow-up of solutions for the weakly coupled systems of parabolic equations on the $n$-dimensional Euclidean space $\mathbb{R}^n$
(see Proposition \ref{Proposition:4.1}).

\begin{Proposition}[{\cite[Lemma 4.1]{bibx:30}}]
\label{Proposition:2.3}
Let $f_0, g_0 > 0$ and $C_p, C_q, T >0$,
and let $p,q >0$ satisfy $pq > 1$.
Let $(f,g) \in ( C^1(\lbrack 0, T))^2$ be non-negative and a solution of
	\[
	\begin{cases}
	\frac{d}{dt} f(t) = (p+1) C_p g(t)^p,
	\quad t \in \lbrack 0,T),\\
	\frac{d}{dt} g(t) = (q+1) C_q f(t)^q,
	\quad t \in \lbrack 0,T),\\
	f(0) = f_0,\quad g(0) = g_0.
	\end{cases}
	\]
Then the identity
	\begin{equation}
	\label{eq:2.1}
	C_q f(t)^{q+1} - C_q f_0^{q+1}
	= C_p g(t)^{p+1} - C_p g_0^{p+1}
	\end{equation}
holds for any $0 \leq t < T$.
Moreover, if $C_q f_0^{q+1} \geq C_p g_0^{p+1}$,
then the following estimates hold:
	\begin{align}
	g(t)
	&\geq \Big\{
	C_p^{\frac{pq-1}{(p+1)(q+1)}} C_q^{-\frac{pq-1}{(p+1)(q+1)}}
	f_0^{-\frac{pq-1}{p+1}}
	- 2^{\frac{-pq}{q+1}} (pq-1) C_p^{\frac{q}{q+1}} C_q^{\frac{1}{q+1}} t
	\Big\}^{- \frac{q+1}{pq-1}}
	\label{eq:2.2} \\
	&- \big( C_p^{-1} C_q f_0^{q+1} - g_0^{p+1} \big)^{\frac{1}{p+1}}
	\quad \mbox{for} \quad 0 < t < T,
	\nonumber \\
	T
	&\leq \frac{2^{\frac{pq}{q+1}}}{pq-1}
	C_p^{-\frac{1}{p+1}} C_q^{-\frac{p}{p+1}}
	f_0^{-\frac{pq-1}{p+1}}.
	\label{eq:2.3}
	\end{align}
We note that
the right-hand side of \eqref{eq:2.3}
is the blow-up time
of the right-hand side of \eqref{eq:2.2}.
\end{Proposition}

For the reader's convenience,
we give a proof of the lemma.

\begin{proof}
Set $F := C_q f^{q+1}\in C^1([0,T))$ and $G := C_p g^{p+1}\in C^1([0,T))$.
Then,
	\[
	F'(t) = (q+1) C_q f(t)^q f'(t)
	= (q+1)(p+1) C_q C_p f(t)^q g(t)^p = G'(t)
	\]
for $t \in [0,T)$,
which implies that the identity $\{F(t)-G(t)\}'=0$ for $t\in [0,T)$. Thus by integrating
it with respect to time over $[0,t)$, we have (\ref{eq:2.1}). Therefore the identities
	\begin{align*}
	f(t)
	&= \left\{ C_q^{-1} C_p g(t)^{p+1} - C_q^{-1} C_p g_0^{p+1} + f_0^{q+1}
	\right\}^{\frac{1}{q+1}},\\
	g(t)
	&= \left\{ C_p^{-1} C_q f(t)^{q+1} - C_p^{-1} C_q f_0^{q+1} + g_0^{p+1}
	\right\}^{\frac{1}{p+1}},\\
	f'(t)
	&= (p+1) C_p \left\{ C_p^{-1} C_q f(t)^{q+1} - C_p^{-1} C_q f_0^{q+1} + g_0^{p+1}
	\right\}^{\frac{p}{p+1}},\\
	g'(t)
	&= (q+1) C_q \left\{ C_q^{-1} C_p g(t)^{p+1} - C_q^{-1} C_p g_0^{p+1} + f_0^{q+1}
	\right\}^{\frac{q}{q+1}}.
	\end{align*}
hold for any $t\in [0,T)$. Since $F(0) \geq G(0)$ by the assumption, the estimate 
\[
   a^{p+1}+b^{p+1}\ge 2^{-p}(a+b)^{p+1}\ \ \text{for}\ a,b\ge 0, 
\]
implies
	\begin{align*}
	g'(t)
	&= (q+1) C_q \left\{ C_q^{-1} C_p g(t)^{p+1} - C_q^{-1} C_p g_0^{p+1} + f_0^{q+1}
	\right\}^{\frac{q}{q+1}}\\
	&= (q+1) C_p^{\frac{q}{q+1}} C_q^{\frac{1}{q+1}}
	\left\{ g(t)^{p+1} + C_p^{-1} C_q f_0^{q+1} - g_0^{p+1}
	\right\}^{\frac{q}{q+1}}\\
	&\geq 2^{\frac{-pq}{q+1}} (q+1) C_p^{\frac{q}{q+1}} C_q^{\frac{1}{q+1}}
	\left\{ g(t) + \big( C_p^{-1} C_q f_0^{q+1} - g_0^{p+1} \big)^{\frac{1}{p+1}}
	\right\}^{\frac{p+1}{q+1} q},
	\end{align*}
for $t\in [0,T)$. Then by solving the ODI and
the triangle inequality, we obtain
	\begin{align}
	&g(t) + \big( C_p^{-1} C_q f_0^{q+1} - g_0^{p+1} \big)^{\frac{1}{p+1}}\notag\\
	&\geq \Big[
	\{ g_0 + \big( C_p^{-1} C_q f_0^{q+1} - g_0^{p+1} \big)^{\frac{1}{p+1}}
	\}^{- \frac{pq-1}{q+1}}
	- 2^{\frac{-pq}{q+1}} (pq-1) C_p^{\frac{q}{q+1}} C_q^{\frac{1}{q+1}} t
	\Big]^{- \frac{q+1}{pq-1}}\notag\\
	&\geq \Big\{
	C_p^{\frac{pq-1}{(p+1)(q+1)}} C_q^{-\frac{pq-1}{(p+1)(q+1)}}
	f_0^{-\frac{pq-1}{p+1}}
	- 2^{\frac{-pq}{q+1}} (pq-1) C_p^{\frac{q}{q+1}} C_q^{\frac{1}{q+1}} t
	\Big\}^{- \frac{q+1}{pq-1}},
	\label{eq:2.4}
	\end{align}
for $t\in [0,T)$, where we have used the identity
	\[
	\frac{p+1}{q+1} q - 1 = \frac{pq-1}{q+1}.
	\]
From the right hand side of (\ref{eq:2.4}), we see that the maximal existence time $T$ of the function $(f,g)$ is estimated by
	\begin{align*}
	T
	&\leq \frac{2^{\frac{pq}{q+1}}}{pq-1} C_p^{-\frac{q}{q+1}} C_q^{-\frac{1}{q+1}}
	C_p^{\frac{pq-1}{(p+1)(q+1)}} C_q^{-\frac{pq-1}{(p+1)(q+1)}}
	f_0^{\frac{1-pq}{p+1}}\\
	&= \frac{2^{\frac{pq}{q+1}}}{pq-1}
	C_p^{-\frac{1}{p+1}} C_q^{-\frac{p}{p+1}}
	f_0^{\frac{1-pq}{p+1}},
	\end{align*}
where we have used the identities
	\begin{align*}
	\frac{pq-1}{(p+1)(q+1)} - \frac{q}{q+1}
	&= - \frac{q+1}{(p+1)(q+1)} = - \frac{1}{p+1},\\
	- \frac{pq-1}{(p+1)(q+1)} - \frac{1}{q+1}
	&= - \frac{pq-1+p+1}{(p+1)(q+1)} = - \frac{p}{p+1}.
	\end{align*}
This completes the proof of the lemma.
\end{proof}

Proposition \ref{Proposition:2.3} gives the main idea
of the proof of Theorem \ref{Theorem:1.1} and Theorem \ref{Theorem:1.2}.
However, unlike the single heat equation with the power-type nonlinearity,
localized average of solutions to \eqref{eq:1.1} with $\mathbb{X}=\mathbb{R}^n$
does not enjoy the ODI of Proposition \ref{Proposition:2.3}
because we are considering the weak coupled nonlinearity.
More precisely, in the case of single heat equation \eqref{eq:1.3},
the localized average of solution to \eqref{eq:1.3}
is a super-solution for the ODI \eqref{eq:1.4},
since the localized average of $- \Delta u$ can be absorbed into the nonlinearity $u^p$.
On the other hand,
in the case of heat system \eqref{eq:1.1} with $\mathbb{X}=\mathbb{R}^n$,
since the nonlinearity is weakly coupled,
the localized averages of each component of the solution with the Laplacian
can not be absorbed.
Thus we use the following modified ODE system
in order to treat the heat systems with $\mathbb{X}=\mathbb{R}^n$
and the weakly coupled nonlinearity (see Section \ref{section:4}).

\begin{Corollary}
\label{Corollary:2.4}
Let $f_0, g_0, \omega > 0$, $p,q >0$ with $pq>1$ and $C_p, C_q >0$.
Let $(f,g) \in \left(C^1([0,T))\right)^2$ be non-negative and a solution for
	\[
	\begin{cases}
	\frac{d}{dt} f(t) + \frac{\omega}{q+1} f(t) = (p+1) C_p g(t)^p,
	\quad t \in \lbrack 0,T),\\
	\frac{d}{dt} g(t) + \frac{\omega}{p+1} g(t) = (q+1) C_q f(t)^q,
	\quad t \in \lbrack 0,T),\\
	f_0 = f_0,\quad g_0 = g_0.
	\end{cases}
	\]
Then the identity
	\[
	C_q f(t)^{q+1} e^{\omega t} - C_q f_0^{q+1}
	= C_p g(t)^{p+1} e^{\omega t} - C_p g_0^{p+1}.
	\]
holds for any $ 0 \leq t < T$.
Moreover, if
	\begin{align}
	\max\bigg(
	2^{\frac{p+1}{q+1} \frac{pq}{pq-1}}
	(q+1)^{- \frac{p+1}{pq-1}}
	(p+1)^{- \frac{p+1}{pq-1}} \omega^{ \frac{p+1}{pq-1}}
	C_p^{-\frac{1}{pq-1}} C_q^{-\frac{p}{pq-1}},
	C_q^{-\frac{1}{q+1}} C_p^{\frac{1}{q+1}} g_0^{\frac{p+1}{q+1}}
	\bigg)
	< f_0,
	\label{eq:2.5}
	\end{align}
then, for $0 \leq t < T$,
	\begin{align*}
	&e^{\frac{\omega}{p+1} t} g(t)\\
	& \geq
	\left\{
	\bigg( \frac{C_p}{C_q} \bigg)^{\frac{pq-1}{(p+1)(q+1)}}
	f_0^{-\frac{pq-1}{p+1}}
	-
	2^{-\frac{pq}{q+1}} (q+1)(p+1) \omega^{-1}
	C_q^{\frac{1}{q+1}} C_p^{\frac{q}{q+1}}
	\left(1 - e^{-\frac{\omega (pq-1)}{(p+1)(q+1)} t}\right)
	\right\}^{-\frac{q+1}{pq-1}}\\
	&- ( C_q C_p^{-1} f_0^{q+1} - g_0^{p+1} )^{\frac{1}{p+1}}
\end{align*}
and
\begin{align*}
	T
	\leq
	- \frac{(q+1)(p+1)}{\omega (pq-1)} \log \bigg(
	1 -
	\frac{2^{\frac{pq}{q+1}}}{(q+1)(p+1)} \omega
	C_q^{-\frac{p}{p+1}} C_p^{-\frac{1}{p+1}}
	f_0^{-\frac{pq-1}{p+1}}
	\bigg)
	\end{align*}
hold.
\end{Corollary}

\begin{proof}
Set $F = C_q f^{q+1}\in C^1([0,T))$ and $G = C_p g^{p+1}\in C^1([0,T))$.
Then by the equation, we have
	\begin{align*}
	F'(t) + \omega F(t)
	&= (q+1) C_q f(t)^q \left( f'(t) + \frac{\omega}{q+1} f \right)\\
	&= (q+1)(p+1) C_q C_p f(t)^q g(t)^p\\
	&= G'(t) + \omega G(t),
	\end{align*}
for any $t\in [0,T)$, which implies that
	\begin{align*}
	F(t) - G(t) = (F(0) - G(0)) e^{-\omega t},
	\end{align*}
for any $t\in [0,T)$. Moreover we obtain
	\begin{align*}
	f(t)
	&= \left\{ C_q^{-1} C_p g(t)^{p+1}
	+ \big(f_0^{q+1} - C_q^{-1} C_p g_0^{p+1} \big) e^{- \omega t}
	\right\}^{\frac{1}{q+1}},\\
	g(t)
	&= \left\{ C_p^{-1} C_q f(t)^{q+1}
	+ \big( g_0^{p+1} - C_p^{-1} C_q f_0^{q+1}\big) e^{-\omega t}
	\right\}^{\frac{1}{p+1}},\\
	f'(t) + \frac{\omega}{q+1} f(t)
	&= (p+1) C_p \left\{ C_p^{-1} C_q f(t)^{q+1}
	+ ( g_0^{p+1} - C_p^{-1} C_q f_0^{q+1} ) e^{-\omega t}
	\right\}^{\frac{p}{p+1}},\\
	g'(t) + \frac{\omega}{p+1} g(t)
	&= (q+1) C_q \left\{ C_q^{-1} C_p g(t)^{p+1}
	+ (f_0^{q+1} - C_q^{-1} C_p g_0^{p+1} ) e^{- \omega t}
	\right\}^{\frac{q}{q+1}},
	\end{align*}
for $t\in [0,T)$. Since we have $F(0) \geq G(0)$ by the assumption of $f_0$, by the estimate $a^{p+1}+b^{p+1}\ge 2^{-p}(a+b)^{p+1}$ for $a,b\ge 0$, we have
	\begin{align*}
	&g'(t) + \frac{\omega}{p+1} g(t)\\
	&= (q+1) C_q \left\{ C_q^{-1} C_p g(t)^{p+1}
	+ \big(f_0^{q+1} - C_q^{-1} C_p g_0^{p+1} \big) e^{- \omega t}
	\right\}^{\frac{q}{q+1}}\\
	&= (q+1) C_q^{\frac{1}{q+1}} C_p^{\frac{q}{q+1}} e^{- \frac{\omega q}{q+1} t}
	\big( e^{\omega t} g(t)^{p+1}
	+ C_q C_p^{-1} f_0^{q+1} - g_0^{p+1}
	\big)^{\frac{q}{q+1}}\\
	&\geq 2^{-\frac{pq}{q+1}} (q+1) C_q^{\frac{1}{q+1}} C_p^{\frac{q}{q+1}}
	e^{- \frac{\omega q}{q+1} t}
	\left\{ e^{\frac{\omega}{p+1} t} g(t)
	+ ( C_q C_p^{-1} f_0^{q+1} - g_0^{p+1} )^{\frac{1}{p+1}}
	\right\}^{\frac{p+1}{q+1}q},
	\end{align*}
for any $t\in [0,T)$. Let
	\[
	\widetilde g(t) = e^{\frac{\omega}{p+1} t} g(t)
	+ ( C_q C_p^{-1} f_0^{q+1} - g_0^{p+1} )^{\frac{1}{p+1}}.
	\]
Then
	\begin{align*}
	\widetilde g(t)'
	&\geq 2^{-\frac{pq}{q+1}} (q+1) C_q^{\frac{1}{q+1}} C_p^{\frac{q}{q+1}}
	e^{ ( \frac{1}{p+1} - \frac{q}{q+1} ) \omega t}
	\widetilde g(t)^{\frac{p+1}{q+1}q}\\
	&= 2^{-\frac{pq}{q+1}} (q+1) C_q^{\frac{1}{q+1}} C_p^{\frac{q}{q+1}}
	e^{ - \frac{pq-1}{(p+1)(q+1)} \omega t}
	\widetilde g(t)^{\frac{p+1}{q+1}q}
	\end{align*}
and therefore
	\[
	\widetilde g(t)
	\geq
	\bigg(
	\widetilde g_0^{-\frac{pq-1}{q+1}}
	- 2^{-\frac{pq}{q+1}} (p+1)(q+1) \omega^{-1}
	C_q^{\frac{1}{q+1}} C_p^{\frac{q}{q+1}}
	(1 - e^{-\frac{pq-1}{(p+1)(q+1)} \omega t})
	\bigg)^{-\frac{q+1}{pq-1}},
	\]
where $\tilde{g}_0 = \tilde{g}(0)$.
We remark that the right hand side of this estimate
blows-up at a certain time because
we assume $f_0$ is sufficiently large.
Since
	\[
	\widetilde g_0
	\geq C_q^{\frac{1}{p+1}} C_p^{-\frac{1}{p+1}} f_0^{\frac{q+1}{p+1}},
	\]
the lifespan is estimated by
	\begin{align*}
	T
	&\leq
	- \frac{(p+1)(q+1)}{\omega (pq-1)} \log \bigg(
	1 -
	\frac{2^{\frac{pq}{q+1}}}{(q+1)(p+1)} \omega
	C_q^{-\frac{1}{q+1}} C_p^{-\frac{q}{q+1}}
	\widetilde g_0^{-\frac{pq-1}{q+1}}
	\bigg)\\
	&\leq
	- \frac{(q+1)(p+1)}{\omega (pq-1)} \log \bigg(
	1 -
	\frac{2^{\frac{pq}{q+1}}}{(q+1)(p+1)} \omega
	C_q^{-\frac{p}{p+1}} C_p^{-\frac{1}{p+1}}
	f_0^{-\frac{pq-1}{p+1}}
	\bigg).
	\end{align*}
\end{proof}

\section{Blow-up of a Weakly Coupled Systems of
Complex Ginzburg-Landau Equations on the Torus}
\label{section:3}

In this section,
we show Theorem \ref{Theorem:1.1}.
Namely, we consider the Cauchy problem
for weakly coupled systems of complex Ginzburg-Landau equations on the torus:
	\begin{align}
	\begin{cases}
	\partial_t u + \alpha_1 \Delta u = \beta_1 |v|^p,
	& t \in \lbrack 0, T), \quad x \in \mathbb T^n,\\
	\partial_t v + \alpha_2 \Delta v = \beta_2 |u|^q,
	& t \in \lbrack 0, T), \quad x \in \mathbb T^n,\\
	u(0,x) = u_0(x),\quad
	v(0,x) = v_0(x),
	& x \in \mathbb T^n,
	\end{cases}
	\label{eq:3.1}
	\end{align}
where $p,q > 0$ with $pq \geq 1$
and $\alpha_1, \alpha_2, \beta_1, \beta_2 \in \mathbb C \backslash \{0\}$.
Here, $u_0, v_0$ are non-zero complex-valued $L^1(\mathbb T^n)$-functions
and $u, v$ are complex-valued unknown functions of $(t,x)$.
We remark that
if $\alpha_1, \alpha_2 \in - i \mathbb R$,
\eqref{eq:3.1} is the weakly coupled systems of Schr\"odinger equations
and
if $\alpha_1, \alpha_2 < 0$,
\eqref{eq:3.1} is the weakly coupled systems of heat equations.
Then Proposition \ref{Proposition:2.3} gives
an apriori upper bound of lifespan and lower bounds of blow-up solution.
The precise statement of Theorem \ref{Theorem:1.1} is the following:

\begin{Proposition}
\label{Proposition:3.1}
Let $u_0, v_0 \in L^1(\mathbb T^n)$ satisfy
	\begin{align*}
	&\frac{|\beta_2|^{2+p}}{q+1}
	(2 \pi)^{-n(p-1)}
	\bigg\{ \mathrm{Re}
	\bigg( \overline \beta_1 \int_{\mathbb T^n} u_0(x) dx \bigg) \bigg\}^{q+1}\\
	&\geq \frac{|\beta_1|^{2+q}}{p+1}
	(2 \pi)^{-n(q-1)}
	\bigg\{ \mathrm{Re}
	\bigg( \overline \beta_2 \int_{\mathbb T^n} v_0(x) dx \bigg) \bigg\}^{p+1}
	>0.
	\end{align*}
If there is a solution $(u,v) \in C^1((0,T),H^2(\mathbb T^n))$
for \eqref{eq:3.1} with the initial datum $(u_0,v_0)$.
Then $(u,v)$ blows up in a finite time.
Moreover, let $T$ be the lifespan.
Then with some positive constants $C_1, C_2, C_3$,
the estimates
	\begin{align*}
	&T \leq T_3
	:= C_1 \bigg\{ \mathrm{Re} \bigg( \overline{\beta_1} \int_{\mathbb T^n} u_0(x)d x \bigg)
	\bigg\}^{\frac{pq-1}{p+1}},\\
	&\mathrm{Re} \bigg( \overline{\beta_1} \int_{\mathbb T^n} u(t,x)d x \bigg)
	\geq C_2 ( T_3 - t)^{-\frac{p+1}{pq-1}},
	\qquad t \in \lbrack 0, T)\\
	&\mathrm{Re} \bigg( \overline{\beta_2} \int_{\mathbb T^n} v(t,x)d x \bigg)
	\geq C_3 ( T_3 - t)^{-\frac{q+1}{pq-1}},
	\qquad t \in \lbrack 0, T).
	\end{align*}
\end{Proposition}

\begin{proof}
Let
	\[
	U(t)
	= \mathrm{Re} \bigg( \overline \beta_1 \int_{\mathbb T^n} u(t,x) dx \bigg),
	\quad
	V(t)
	= \mathrm{Re} \bigg( \overline \beta_2 \int_{\mathbb T^n} v(t,x) dx \bigg).
	\]
Multiplying $\overline \beta_1$
by the both sides of the first equation of \eqref{eq:3.1},
integrating over $\mathbb T^n$,
taking real part,
and by the H\"older inequality,
	\[
	\frac{d}{dt} U(t)
	= |\beta_1|^{2} |\beta_2|^{-p} \int_{\mathbb T^n} |\overline \beta_2 v(t,x)|^p dx
	\geq |\beta_1|^{2} |\beta_2|^{-p} (2 \pi)^{-n(p-1)} V(t)^p.
	\]
Here we have used the condition that $p \geq 1$ and the fact that
	\[
	\int_{\mathbb T^n} \Delta f(x) dx = 0
	\]
for $f \in H^2(\mathbb T^n)$.
Similarly,
	\[
	\frac{d}{dt} V(T)
	\geq |\beta_2|^{2} |\beta_1|^{-q} (2 \pi)^{-n(q-1)} U(t)^q.
	\]
Then by Proposition \ref{Proposition:2.3}
with
	\[
	C_p = \frac{|\beta_1|^{2} |\beta_2|^{-p} (2 \pi)^{-n(p-1)}}{p+1},\qquad
	C_q = \frac{|\beta_2|^{2} |\beta_1|^{-q} (2 \pi)^{-n(q-1)}}{q+1},
	\]
we obtain Proposition \ref{Proposition:3.1}.
\end{proof}

\section{Blow-up of Weakly Coupled Systems of
Complex Ginzburg-Landau Equations on the Euclidean Space}
\label{section:4}

In this section,
we show Theorem \ref{Theorem:1.2}.
To state the precise statement of Theorem \ref{Theorem:1.2},
we introduce some notation.
Let $B(r) \subset \mathbb R^n$ be the open ball centered at the origin with radius $r$.
Let $\phi \in C^2(B(1))$ be a non-negative function satisfying
	\[
	\begin{cases}
	- \Delta \phi \leq \lambda \phi, &\mbox{in} \quad B(1),\\
	\phi =0, &\mbox{on} \quad \partial B(1),\\
	\nabla \phi =0, &\mbox{on} \quad \partial B(1)
	\end{cases}
	\]
for some $\lambda > 0$.
We remark that such $\phi$ exists.
For example, let $\psi$ be an eigenfunction of $-\Delta$
on $B(1)$ with the Dirichlet condition
with a positive eigenvalue $\lambda$.
Then $\psi^2$ also satisfies the Neumann condition and
	\[
	-\Delta (\psi^2)
	= 2 \psi (-\Delta \psi) - 2 |\nabla \psi|^2
	\leq 2 \lambda \psi^2.
	\]
Let $p,q > 0$, $pq > 1$ and $U_0, V_0 > 0$.
Put
	\begin{align*}
	R_0
	&= R_0(p,q,\phi,\lambda,\beta_1,\beta_2,U_0,V_0)
	= \max(R_1,R_2),\\
	R_1
	&=R_1(p,q,\phi,\lambda,\beta_1,\beta_2,U_0)\\
	&= 2^{\frac{pq}{q+1} \frac{1}{2-n \frac{pq-1}{p+1}}}
	\bigg(\frac{p+1}{q+1}\bigg)^{\frac{1}{pq-1} \frac{1}{2 \frac{p+1}{pq-1} -n}}
	( \max(|\alpha_1|,|\alpha_2|)\lambda)^{\frac{1}{2 -n\frac{pq-1}{p+1}}}\\
	&\cdot |\beta_1|^{\frac{pq-2}{2 (p+1)-n(pq-1)}}
	|\beta_2|^{-\frac{p}{2 (p+1)-n(pq-1)}}
	U_0^{-\frac{1}{2 \frac{p+1}{pq-1}-n}},\\
	R_2
	&=
	R_2(p,q,\phi,\beta_1,\beta_2,U_0,V_0)\\
	&=
	\bigg(\frac{q+1}{p+1}\bigg)^{\frac{1}{n(p-q)}}
	|\beta_1|^{\frac{q+2}{n(p-q)}}
	|\beta_2|^{-\frac{2+p}{n(p-q)}}
	\| \phi \|_{L^1(B(1))}^{-\frac{1}{n}}
	V_0^{\frac{p+1}{n(p-q)}}
	U_0^{-\frac{q+1}{n(p-q)}}.
	\end{align*}
Moreover, we define positive constants $C_1,C_2,C_3$ as follows:
	\begin{align}
	C_1
	&=C_1(p,q,\phi,\beta_1,\beta_2)
	\label{eq:4.1}\\
	&= \bigg( \frac{q+1}{p+1} \bigg)^{\frac{pq-1}{(p+1)(q+1)}}
	\bigg(
	\frac{|\beta_1|^{2+q}}%
	{|\beta_2|^{2+p}}
	\bigg)^{\frac{pq-1}{(p+1)(q+1)}}
	\|\phi\|_{L^1(B(1))}^{-\frac{(pq-1)(p-q)}{(p+1)(q+1)}},
	\nonumber\\
	C_2
	&= C_2(p,q,\phi,\beta_1,\beta_2)
	\label{eq:4.2}\\
	&=
	2^{-\frac{pq}{q+1}} \bigg( \frac{q+1}{p+1} \bigg)^{\frac{q}{q+1}}
	|\beta_1|^{\frac{q}{q+1}}
	|\beta_2|^{\frac{2-pq}{q+1}}
	\| \phi \|_{L^1(B(1))}^{-\frac{pq-1}{q+1}},
	\nonumber\\
	C_3
	&=C_3(p,q,\beta_1,\beta_2)
	\nonumber\\
	&= 2^{\frac{pq}{q+1} \frac{1}{1-\frac{n}{2} \frac{pq-1}{p+1}}}
	\frac{q+1}{pq-1}
	\bigg( \frac{p+1}{q+1}
	\bigg)^{\frac{1}{p+1} \frac{1}{1- \frac{n}{2} \frac{pq-1}{p+1}}}
	\nonumber\\
	&\cdot
	\max(|\alpha_1|,|\alpha_2|)^{\frac{n}{2} \frac{1}{\frac{p+1}{pq-1}-\frac{n}{2}}}
	|\beta_1|^{-\frac{2(2-pq)}{2 (p+1)-n(pq-1)}}
	|\beta_2|^{-\frac{2p}{2 (p+1)-n(pq-1)}}.
	\nonumber
	\end{align}

Now we are in position to state our main statement more precisely.
\begin{Proposition}
\label{Proposition:4.1}
Let $\alpha_1, \alpha_2 <0$, $\beta_1, \beta_2 \in \mathbb C \backslash \{0\}$.
Let $p,q \geq 1$ satisfy $p \geq q$ and $\frac{p+1}{pq-1} > \frac{n}{2}$.
Let $\phi$ be the function defined above.
Let $U_0, V_0$ be positive numbers.
Let $u_0$ and $v_0$ be $L_{\mathrm{loc}}^1$-functions
satisfying
	\[
	U_0
	\leq \mathrm{Re} \bigg( \overline \beta \int_{B(R)} u_0(x) \phi\bigg(\frac{x}{R}\bigg) dx \bigg),\quad
	V_0
	\leq \mathrm{Re} \bigg( \overline \beta \int_{B(R)} v_0(x) \phi\bigg(\frac{x}{R}\bigg) dx \bigg)
	\]
for $R > R_0$ with $R_0$ above and
	\[
	\overline \beta_1 u_0(x),\quad
	\overline \beta_2 v_0(x) > 0,
	\qquad \mbox{a.e.} \quad x \in \mathbb R^n
	\]
If there exists a solution $(u,v) \in (C^2((0,T) \times \mathbb R^n))^2$
for \eqref{eq:1.1},
then
	\begin{align*}
	&\int_{B(R)} v(t,x) \phi\bigg(\frac{x}{R}\bigg) dx
	+ \bigg( \frac{p+1}{q+1} \|\phi\|_{L^1(B(1))}^{p-q} R^{n(p-q)} U_0^{q+1}
	- V_0^{p+1} \bigg)^{\frac{1}{p+1}}\\
	&\geq \Big\{
	C_1
	R^{-\frac{n(pq-1)(p-q)}{(p+1)(q+1)}}
	U_0^{-\frac{pq-1}{p+1}}
	- C_2 \lambda^{-1} R^{2-n\frac{pq-1}{q+1}} (1 - e^{- \frac{pq-1}{q+1} \lambda R^{-2}t} )
	\Big\}^{- \frac{q+1}{pq-1}},
	\end{align*}
for $R > R_0$ and $0 \leq t < T$.
Moreover, the lifespan $T$ is estimated by
	\begin{align*}
	T
	&\leq C_3
	\min \bigg\{ - \widetilde R^2 \log ( 1 - \widetilde R^{-2+n\frac{pq-1}{p+1}})
	\ ;\ \widetilde R \geq \frac{R_0}{R_1} \bigg\}\\
	&\cdot \lambda^{\frac{n}{2} \frac{1}{\frac{p+1}{pq-1}-\frac{n}{2}}}
	\| \phi \|^{\frac{1}{\frac{p+1}{pq-1} - \frac{n}{2}}}
	U_0^{-\frac{1}{\frac{p+1}{pq-1} -\frac{n}{2}}},
	\end{align*}
\end{Proposition}

\begin{proof}
Since $(u,v)$ is a solution,
by letting
	\[
	\begin{cases}
	U(t)
	= \mathrm{Re} \bigg( \overline \beta_1
	\int_{\mathbb R^n} u(t,x) \phi (\frac{x}{R}) dx \bigg),\\
	V(t)
	= \mathrm{Re} \bigg( \overline \beta_2
	\int_{\mathbb R^n} v(t,x) \phi (\frac{x}{R}) dx \bigg),
	\end{cases}
	\]
the following ODI holds as long as $U(t), V(t) \geq 0$:
	\[
	\begin{cases}
	U'(t) + \max( |\alpha_1|,|\alpha_2|)
	\frac{p+1}{q+1}\lambda R^{-2} U(t)\\
	\geq U'(t) + |\alpha_1| \lambda R^{-2} U(t)
	\geq R^{-n(p-1)} \| \phi \|_{L^1(\mathbb R^n)}^{-p+1}
	|\beta_1|^{2}|\beta_2|^{-p} V(t)^p,\\
	V'(t) + \max( |\alpha_1|,|\alpha_2|) \frac{p+1}{p+1}\lambda R^{-2} V(t)\\
	\geq V'(t) + |\alpha_2| \lambda R^{-2} V(t)
	\geq R^{-n(q-1)} \| \phi \|_{L^1(\mathbb R^n)}^{-q+1} |\beta_2|^{2}|\beta_1|^{-q} U(t)^q,
	\end{cases}
	\]
where $\alpha_1$ and $\alpha_2 \leq 0$.
Let $\widetilde \lambda = \max(|\alpha_1|,|\alpha_2|) \lambda$.
By applying Corollary \ref{Corollary:2.4} with
$\omega = (p+1) \widetilde \lambda R^{-2}$ and
	\[
	C_p = \frac{\| \phi \|_{L^1(B(1))}^{-p+1}}{p+1}
	|\beta_1|^{2} |\beta_2|^{-p} R^{-n(p-1)},\quad
	C_q = \frac{\| \phi \|_{L^1(B(1))}^{-q+1}}{q+1}
	|\beta_2|^{2} |\beta_1|^{-q} R^{-n(q-1)},
	\]
$U$ and $V$ blow up if \eqref{eq:2.5} is satisfied.
\eqref{eq:2.5} is rewritten by
	\[
	U_0
	>
	\max\bigg(
	\mu_1 R^{-2 \frac{p+1}{pq-1}+n}
	,\ 
	\mu_2 R^{-n\frac{p-q}{q+1}}
	V_0^{\frac{p+1}{q+1}}
	\bigg),
	\]
where
	\begin{align*}
	\mu_1 &= \mu_1(\beta_1,\beta_2,p,q,\widetilde \lambda,\phi)\\
	&=
	2^{\frac{p+1}{q+1} \frac{pq}{pq-1}}
	\bigg(\frac{p+1}{q+1}\bigg)^{\frac{1}{pq-1}}
	\widetilde \lambda^{ \frac{p+1}{pq-1}}
	|\beta_1|^{1-\frac{1}{pq-1}}
	|\beta_2|^{-\frac{p}{pq-1}}
	\| \phi \|_{L^1(B(1))},\\
	\mu_2
	&=\mu_2(\beta_1,\beta_2,p,q,\phi)\\
	&=
	\bigg(\frac{q+1}{p+1}\bigg)^{\frac{1}{q+1}}
	|\beta_1|^{1+\frac{1}{q+1}}
	|\beta_2|^{-\frac{2+p}{q+1}}
	\| \phi \|_{L^1(B(1))}^{-\frac{p-q}{q+1}}.
	\end{align*}
Here
	\[
	U_0 > \mu_1 R^{-2 \frac{p+1}{pq-1}+n}
	\]
for $R > R_1$,
since $-2 \frac{p+1}{pq-1}+n < 0$
and
	\begin{align*}
	&\mu_1(p,q,\widetilde \lambda,\phi)^{\frac{1}{2 \frac{p+1}{pq-1}-n}}
	U_0^{-\frac{1}{2 \frac{p+1}{pq-1}-n}}\\
	&= 2^{\frac{pq}{q+1} \frac{1}{2-n \frac{pq-1}{p+1}}}
	\bigg(\frac{p+1}{q+1}\bigg)^{\frac{1}{pq-1} \frac{1}{2 \frac{p+1}{pq-1} -n}}
	\widetilde \lambda^{\frac{1}{2 -n\frac{pq-1}{p+1}}}\\
	&\cdot |\beta_1|^{\frac{pq-2}{2 (p+1)-n(pq-1)}}
	|\beta_2|^{-\frac{p}{2 (p+1)-n(pq-1)}}
	U_0^{-\frac{1}{2 \frac{p+1}{pq-1}-n}} = R_1.
	\end{align*}
Similarly,
	\[
	U_0
	>
	\mu_2 R^{-n\frac{p-q}{q+1}}
	V_0^{\frac{p+1}{q+1}}
	\]
for $R > R_2$,
since
	\begin{align*}
	&\mu_2(p,q,\phi)^{\frac{q+1}{n(p-q)}}
	V_0^{\frac{p+1}{n(p-q)}}
	U_0^{-\frac{q+1}{n(p-q)}}\\
	&=
	\bigg(\frac{q+1}{p+1}\bigg)^{\frac{1}{n(p-q)}}
	|\beta_1|^{\frac{q+2}{n(p-q)}}
	|\beta_2|^{-\frac{2+p}{n(p-q)}}
	\| \phi \|_{L^1(B(1))}^{-\frac{1}{n}}
	V_0^{\frac{p+1}{n(p-q)}}
	U_0^{-\frac{q+1}{n(p-q)}} = R_2.
	\end{align*}
\eqref{eq:4.1} and \eqref{eq:4.2} are calculated as
	\begin{align*}
	&C_q^{-\frac{pq-1}{(p+1)(q+1)}} C_p^{\frac{pq-1}{(p+1)(q+1)}}
	U_0^{-\frac{pq-1}{p+1}},\\
	&= \bigg( \frac{q+1}{p+1} \bigg)^{\frac{pq-1}{(p+1)(q+1)}}
	\|\phi\|_{L^1(B(1))}^{-\frac{(pq-1)(p-q)}{(p+1)(q+1)}}
	\bigg(
	\frac{|\beta_1|^{2+q}}%
	{|\beta_2|^{2+p}}
	\bigg)^{\frac{pq-1}{(p+1)(q+1)}}
	R^{-\frac{n(pq-1)(p-q)}{(p+1)(q+1)}}
	U_0^{-\frac{pq-1}{p+1}},\\
	&=
	C_1
	R^{-\frac{n(pq-1)(p-q)}{(p+1)(q+1)}}
	U_0^{-\frac{pq-1}{p+1}},\\
	&2^{-\frac{pq}{q+1}} (q+1)(p+1) \omega^{-1}
	C_q^{\frac{1}{q+1}} C_p^{\frac{q}{q+1}}\\
	&=
	2^{-\frac{pq}{q+1}} \bigg( \frac{q+1}{p+1} \bigg)^{\frac{q}{q+1}}
	|\beta_1|^{\frac{q}{q+1}}
	|\beta_2|^{\frac{2-pq}{q+1}}
	\| \phi \|_{L^1(B(1))}^{-\frac{pq-1}{q+1}}
	\widetilde \lambda^{-1}
	R^{2-n\frac{pq-1}{q+1}},\\
	&=
	C_2 \widetilde \lambda^{-1}
	R^{2-n\frac{pq-1}{q+1}}.
	\end{align*}
Moreover,
	\begin{align*}
	T
	&\leq
	- \min \bigg\{ \frac{(q+1)(p+1)}{\omega (pq-1)} \log \bigg(
	1 -
	\frac{2^{\frac{pq}{q+1}}}{(q+1)(p+1)} \omega
	C_q^{-\frac{p}{p+1}} C_p^{-\frac{1}{p+1}}
	U_0^{-\frac{pq-1}{p+1}} \bigg)
	;\ R > R_0
	\bigg\}\\
	&=
	- \min \bigg\{ \frac{(q+1)}{\widetilde \lambda (pq-1)} R^2
	\log \bigg( 1 -
	\bigg( \frac{R}{R_1} \bigg)^{-2 + n \frac{pq-1}{p+1}} \bigg)
	;\ R > R_0
	\bigg\}
	\end{align*}
Let $\widetilde R = R/ R_1$.
Then
	\begin{align*}
	T
	&\leq
	2^{\frac{pq}{q+1} \frac{1}{1-\frac{n}{2} \frac{pq-1}{p+1}}}
	\frac{q+1}{pq-1}
	\bigg( \frac{p+1}{q+1} \bigg)^{\frac{1}{p+1} \frac{1}{1- \frac{n}{2} \frac{pq-1}{p+1}}}
	\widetilde \lambda^{\frac{1}{1-\frac{n}{2}\frac{pq-1}{p+1}}-1}
	\| \phi \|^{\frac{1}{\frac{p+1}{pq-1} - \frac{n}{2}}}
	U_0^{-\frac{1}{\frac{p+1}{pq-1} -\frac{n}{2}}}\\
	&\cdot
	|\beta_1|^{-\frac{2(2-pq)}{2 (p+1)-n(pq-1)}}
	|\beta_2|^{-\frac{2p}{2 (p+1)-n(pq-1)}}
	\min \bigg\{ - \widetilde R^2 \log ( 1 - \widetilde R^{-2+n\frac{pq-1}{p+1}})
	\ ;\ \widetilde R \geq \frac{R_0}{R_1} \bigg\}\\
	&\leq
	C_3
	\lambda^{\frac{n}{2} \frac{1}{\frac{p+1}{pq-1}-\frac{n}{2}}}
	\| \phi \|^{\frac{1}{\frac{p+1}{pq-1} - \frac{n}{2}}}
	U_0^{-\frac{1}{\frac{p+1}{pq-1} -\frac{n}{2}}}\\
	&\cdot \min \bigg\{ - \widetilde R^2 \log ( 1 - \widetilde R^{-2+n\frac{pq-1}{p+1}})
	\ ;\ \widetilde R \geq \frac{R_0}{R_1} \bigg\}.
	\end{align*}
Here, we have used the fact that
$f(x) = -x \log(1-x^{-\theta})$
for $0 < \theta < 1$ attains the minimum on
$\lbrack R_0/R_1,\infty) \subset \lbrack 1,\infty)$.
Indeed,
	\[
	f'(x) = - \log(1-x^{-\theta}) - \theta \frac{x^{-\theta}}{1-x^{-\theta}}.
	= \alpha - \log(1-x^{-\theta}) - \theta \frac{1}{1-x^{-\theta}}.
	\]
by putting $y=-\log(1-x^{-\theta}) \in (0,\infty)$,
$f'(x) = g(y) = \theta + y -\theta e^y$
and therefore $f$ attains the minimum.
\end{proof}


\end{document}